\documentclass[12pt]{amsart}
\usepackage{amsmath,amssymb,amscd,graphicx,verbatim,url}
\usepackage{hyperref}
\pdfoutput=1
\usepackage{mathpazo}
\usepackage{graphicx}
\newcommand{\tre}{\text{Re}}

\theoremstyle{plain}

\theoremstyle{definition}

\newtheorem*{theorem}{Theorem}
\newtheorem*{proposition}{Proposition}
\newtheorem*{lemma}{Lemma}
\newtheorem*{remark}{Remark}

\begin{document}
\title[Notes on Conrey-Iwaniec]{On the theorem\\ of Conrey and Iwaniec}
\author{Jeffrey Stopple}
\address{Department of Mathematics\\University of California, Santa Barbara\\Santa Barbara, CA 93106-3080}
\email{stopple@math.ucsb.edu}
\subjclass[2000]{11M20,11R29}
\begin{abstract}
An exposition on \lq Spacing of zeros of Hecke L-functions and the class number problem\rq\ by Conrey and Iwaniec; any errors are my own.
\end{abstract}

\maketitle

\section*{History}  In number theory, binary quadratic forms $ax^2+bxy+cy^2=Q(x,y)$ have a long and glorious history.  Motivated by Fermat's and Euler's questions about the range of such functions (Which odd primes $p$ can be written as $p=x^2+y^2$?  Or generally $p=x^2+ny^2$?), one defines equivalence relation which identifies forms with the same range:
\[
Q^\prime\sim Q\quad\text{if}\quad Q^\prime(x,y)=Q((x,y)\gamma)
\]
for some $\gamma$ in $SL(2,\mathbb Z)$. The form $Q$ is positive definite if the discriminant $-D=b^2-4ac<0$; a calculation shows the discriminant is an invariant of the equivalence class.  Gauss showed the classes of fixed discriminant $-D$ form an abelian group $\mathcal C(-D)$ of finite order $h(-D)$, and speculated \cite[Art. 303]{Gauss} about the growth of this class number as a function of the discriminant:
\begin{quote}
\ldots  the series of  [discriminants] corresponding to the same given classification (i.e. the given number of both genera and classes) always seems to terminate with a finite number\ldots     However \emph{rigorous} proofs of these observations seem to be very difficult.
\end{quote}
Dirichlet's analytic class number formula tells us
\begin{equation}
\label{Eq:CNF}
L(1,\chi_{-D})=\frac{\pi h(-D)}{D^{1/2}},
\end{equation}
where $L(s,\chi_{-D})$ is defined, for $\tre(s)>1$ by
\[
L(s,\chi_{-D})=\sum_{n=1}^\infty \chi_{-D}(n)n^{-s}.
\]
The Kronecker symbol $\chi_{-D}$ is equal, in the case $D=q$ an odd prime, to the Jacobi symbol
\[
\chi_{-q}(a)=\left(\frac{a}{q}\right)=
\begin{cases}
 0,&\text{if }q | a\\
1,&\text{if }a\equiv \square \bmod q\\
-1,&\text{if }a\text{ is not }\equiv \square \bmod q.
\end{cases}
\]
The Kronecker symbol extends the Jacobi symbol multiplicatively in the \lq denominator\rq\ for odd primes, and no longer detects squares modulo $D$.  It has instead the important property that for primes $p$ not dividing $D$, $\chi_{-D}(p)=1$ if and only if $p=Q(x,y)$ for \emph{some} form $Q$ of discriminant $-D$.

With the analytic class number formula in hand, one can show the class number $h(-D)$ is not be much larger than $D^{1/2}$.  In fact 
\[
h(-D)\ll_\epsilon D^{1/2+\epsilon}\qquad\text{ for any }\epsilon>0.
\]
For lower bounds, Gauss' observation above still holds.  Via a continuity argument, if $L(\beta,\chi_{-D})=0$ for some real $\beta$ near $s=1$ then one expects $L(1,\chi_{-D})$ would be small, and so $h(-D)$ would be small compared to $D^{1/2}$.  Such a \lq Landau-Siegel zero\rq\  would of course violate the Generalized Riemann Hypothesis.  Before 1918, Hecke realized one could show the converse: in the absence of a Landau-Siegel zero one has
\[
 D^{1/2+\epsilon}\ll_\epsilon h(-D)\qquad\text{ for any }\epsilon>0.
\]
Siegel was then able to show that for any $\epsilon>0$, there is an ineffective constant $C(\epsilon)$, such that
\begin{equation}\label{Eq:SiegelThm}
C(\epsilon)D^{1/2-\epsilon}<h(-D).
\end{equation}
\lq Ineffective\rq\ means that the proof of the theorem bifurcates depending on whether or not the Generalized Riemann Hypothesis (GRH) is true or false.  Thus determining a value for the constant requires resolving the GRH!
Unconditionally the best we can do is the Goldfeld-Gross-Zagier lower bound
\begin{equation}\label{Eq:GGZ}
\log(D)\ll h(-D),
\end{equation}
with an explicitly known implied constant.

\section*{Introduction}  In \cite{CI}, Conrey and Iwaniec deduce lower bounds on the class number $h(-D)$, on the hypothesis that (roughly speaking) sufficiently many zeros of various $L$-functions attached to characters of the class group $\mathcal C(-D)$ are sufficiently closely spaced.  In particular, for a special case one may consider the zeros of $\zeta(s)L(s,\chi_{-D})$, where $\zeta(s)$ is the Riemann zeta function.  

The theorem referred to is the culmination of a long history of folklore, announced, and unpublished results \cite{HB,Montgomery,M2,Stark,Watkins} regarding the Deuring-Heilbronn phenomenon.   In this situation, the existence of a small class number (or equivalently, a Landau-Siegel zero of $L(s,\chi_{-D})$) forces very many zeros of $\zeta(s)$ to be very regularly spaced instead.   
The difficulty others have had in getting a result in a publishable form attests to the accomplishment of Conrey and Iwaniec.  Much of their paper is devoted to a technical result \cite[Proposition 9.2]{CI} (see (\ref{Eq:9.12}) below.)  

This expository note is aimed at explaining the consequences of that proposition.  We include a modest investigation of the spacing of the lowest $10^7$ zeros of $\zeta(s)$.
 In the last section we investigate what their conditional bound has to say about the order $p(-D)$ of the principal genus.  

\section*{Exposition}  Our starting point is Proposition 9.2 of \cite{CI}, which we briefly summarize.
\begin{proposition}[Conrey-Iwaniec] Let $\mathcal S$ be a set of zeros of $\zeta(s)$ on the critical line $\rho=1/2+it$ with $2\le t\le T$ which are spaced by at least one, and let $\rho^\prime=1/2+it^\prime$ be the nearest zero on the critical line with $t<t^\prime$.  Then
\begin{multline}\label{Eq:9.12}
\sum_\rho\left|\frac{\sin(t-t^\prime)\log(t)}{(t-t^\prime)\log(t)}\right|\ll\\
\frac{T}{\log T}(\log D)^6+T\log T L(1,\chi_{-D})^{1/2}(\log D)^3
\end{multline}
where the constant implied by the $\ll$ is absolute.  (For simplicity we have specified the $L$-function to be $\zeta(s)L(s,\chi_{-D})$, and the points in the set $\mathcal S$ to be zeros, which eliminates the third term on the right side of \cite[(9.12)]{CI}.)  
\end{proposition}
\begin{remark}
Observe how striking this is: the left side depends only on the zeros of $\zeta(s)$, while the right side depends only on the discriminant $-D$!
\end{remark}

We now specify more precisely a choice of $\mathcal S$ as in \S 10 of \cite{CI}.  Let $\alpha$ denote $1/\sqrt{\log T}$, and assume about the $\rho$ in $\mathcal S=\mathcal S(\alpha,T)$ that
\begin{equation}\label{defs}
|t-t^\prime|\le\frac{\pi(1-\alpha)}{\log t}.
\end{equation}
\begin{lemma}
\begin{multline}\label{Eq:3}
\sharp S(\alpha,T)\ll\frac{T}{(\log T)^{1/2}}(\log D)^6+\\
T(\log T)^{3/2} (\log D)^3 D^{-1/4}h(-D)^{1/2}.
\end{multline}
\end{lemma}
\begin{proof}
One easily deduces that for $0<x<1$, $1-x<\sin(\pi x)/(\pi x)$.  Now take $x=(t-t^\prime)\log (t)/\pi$;
from (\ref{defs}) we deduce that
\[
\alpha \le \left|\frac{\sin(t-t^\prime)\log(t)}{(t-t^\prime)\log(t)}\right|.
\]
So $\sharp S(\alpha,T)\cdot\alpha$ is less than the right hand side of (\ref{Eq:9.12}) above.
Plug  Dirichlet's analytic class number formula (\ref{Eq:CNF}) into (\ref{Eq:9.12}) to deduce that
\[
\sharp S(\alpha,T)\cdot\alpha\ll\frac{T}{\log T}(\log D)^6+
T\log T (\log D)^3 D^{-1/4}h(-D)^{1/2}.
\]
(We have absorbed a term $\sqrt{\pi}$ into the universal constant implied by the $\ll$.)  Setting $\alpha=1/\sqrt{\log T}$ gives (\ref{Eq:3}).
\end{proof}

In \cite{CI} the next step is to make each of the two terms on the right side of (\ref{Eq:3}) less than $ T/(\log T)^{\delta}$, where $0<\delta<1/2$ is a free parameter.  For concreteness, we will now choose $\delta=1/5$.  Thus, we desire
\[
(\log D)^6<(\log T)^{3/10}\quad\text{or}\quad (\log D)^{20}< \log T,
\]
and also that
\[
(\log T)^{17/10}<(\log D)^{-3} D^{1/4}h(-D)^{-1/2},
\]
i.e.
\[
\log T<(\log D)^{-30/17} D^{5/34}h(-D)^{-5/17}.
\]
Thus, on the interval
\begin{equation}\label{Eq:4}
 (\log D)^{20}< \log T
 <(\log D)^{-30/17} D^{5/34}h(-D)^{-5/17},
\end{equation}
we have that
\begin{equation}\label{almost}
\sharp S(\alpha,T)\ll\frac{T}{(\log T)^{1/5}}
\end{equation}
(where we have absorbed a factor of $2$ into the universal constant implied by the $\ll$.)  

\begin{remark} This is unconditional.  However, a little arithmetic show that the interval (\ref{Eq:4}) is nonempty only when
\begin{equation}\label{Eq:CIBound}
h(-D)<\frac{D^{1/2}}{\log(D)^{74}}.
\end{equation}
The reader should compare this inequality with that of Siegel's Theorem (\ref{Eq:SiegelThm}).  (Note also the right side above is not greater than $1$ until $10^{445}<D$.)
\end{remark}

Again following \S10 of \cite{CI} we will next drop the side condition that the $\rho$ in $\mathcal S(\alpha,T)$ are spaced by one, at a cost of an extra factor of $\log T$ in the bound.  Specifically, let $\mathcal R(T)$ be the set of zeros of $\zeta(s)$ on the critical line $\rho=1/2+it$ with $2\le t\le T$ such that
\begin{equation}\label{close}
|t-t^\prime|\le\frac{\pi}{\log t}\left(1-\frac{1}{\sqrt{\log t}}\right)
\end{equation}
where $\rho^\prime=1/2+it^\prime$ is the nearest zero on the critical line with $t<t^\prime$.  (The average gap between consecutive zeros is $2\pi/\log t$, so $\mathcal R(T)$ selects zeros with slightly less than half the average gap.)
\begin{theorem}  For $T$ in the same interval defined by (\ref{Eq:4}), i.e.
\[
 (\log D)^{20}< \log T
 <(\log D)^{-30/17} D^{5/34}h(-D)^{-5/17},
\]
we have
\begin{equation}\label{main}
\sharp \mathcal R(T) \ll T (\log T)^{4/5}.
\end{equation}
where the constant implied by the $\ll$ is universal.
\end{theorem}
\begin{proof}
The total number of zeros $\rho$ of $\zeta(s)$ in an interval $[t,t+1)$ of length $1$ is $O(\log t)$  \cite[Theorem 9.2]{Tit}.  Then (\ref{main}) follows from (\ref{almost}) and the pigeonhole principle.  That is, if (\ref{main}) failed to be true, since no more than $O(\log t)$ of these zeros can be in any one unit interval, we would be able to select $\gg$ $T/(\log T)^{1/5}$ well spaced zeros.  Since
\[
t<T\quad \Rightarrow\quad 1-\frac{1}{\sqrt{\log t}}<1-\frac{1}{\sqrt{\log T}},
\]
these zeros lie in $\mathcal S(\alpha,T)$, and this violates (\ref{almost}).
\end{proof}
\section*{Discussion}
There are strong results \cite{D,H,Stark,Stark1} showing that a Landau-Siegel zero arising from a discriminant with $h(-D)=1$ would influence the spacing of the zero of $\zeta(s)$, tending to make them nearly periodic (over a certain range.)\ \ 
Motivated by this, Montgomery began to investigate the spacing in general and was able to give evidence that the \lq pair correlation\rq\ for the (not necessarily consecutive) zeros was
\[
1-\left(\frac{\sin\pi x}{\pi x}\right)^2.
\]
(To be more precise, he proved this for test functions which have Fourier transform supported in $[-2,2]$)\ \ 
A lucky encounter with Freeman Dyson lead to the realization this was the pair correlation for eigenvalues of random hermitian matrices.  Random Matrix Theory, previously applied by physicists to model problems in nuclear physics, now had applications in number theory.  (See \cite{Firk} for an exposition.)
This same theory suggests that the renormalized gaps between \emph{consecutive} zeros should be distributed according to the Gaudin distribution.  This is complicated to compute \cite{FO}, but well approximated by the function called the \lq Wigner surmise\rq\  (in the GUE case)
\[
p(x)=\frac{32}{\pi^2} x^2 \exp\left(-\frac{4}{\pi} x^2\right),
\]
see \cite[(202) p. 55]{Porter}.
Figure \ref{F:plot} shows the graph of $p(x)$ (solid), and a power series approximation to the Gaudin distribution, to 30 terms (dotted).  Both are in good agreement out to $x=1/2$.  In \emph{Mathematica} we easily compute that
\[
\int_0^{1/2}p(x)\, dx=0.11200,
\]
while the power series approximation give an estimate of  $0.11097$ for the corresponding integral.

The total number $N(T)$ of zeros of $\zeta(s)$ of height $\le T$  satisfies \cite[Theorem 9.4]{Tit}
\[
N(T)\sim \frac{T}{2\pi}\log(T).
\]
Since we conjecture that about $11\%$ of the gaps \lq should\rq\ be less than half the average, 
\[
\sharp \mathcal R(T) \text{ \lq should\rq\ be }\sim \, 0.11 \cdot \frac{T}{2\pi}\log(T).
\]
This contradicts (\ref{main}) if $T$ (which is determined by $D$) is sufficiently large.  What \lq sufficiently large\rq\ means here is determined by the size of the uncomputed universal constant in the $\ll$ symbol.  To be more precise, denote this unknown constant by $C$.  Then 
\[
C\cdot T\log(T)^{4/5} < 0.11 \cdot \frac{T}{2\pi}\log(T)
\]
exactly when
\begin{equation}\label{Eq:C}
\exp\left((2\pi C/0.11)^5\right)<T.
\end{equation}
Conrey and Iwaniec conclude by forming the contrapositive statement: Roughly speaking, if (as Random Matrix Theory predicts) (\ref{main}) fails to be true, then 
\begin{equation}\label{Eq:contra}
\frac{D^{1/2}}{\log(D)^{74}}\le h(-D).
\end{equation}
But again, the contradiction is avoided only for $D$ large enough so that the corresponding $T$ satisfies (\ref{Eq:C}).  Via the first inequality of (\ref{Eq:4}), we see that
\[
\exp\left(\left(2\pi C/0.11\right)^{1/4}\right)<D
\]
suffices.
\begin{figure}
\begin{center}
\includegraphics[scale=.8, viewport=0 0 320 240,clip]{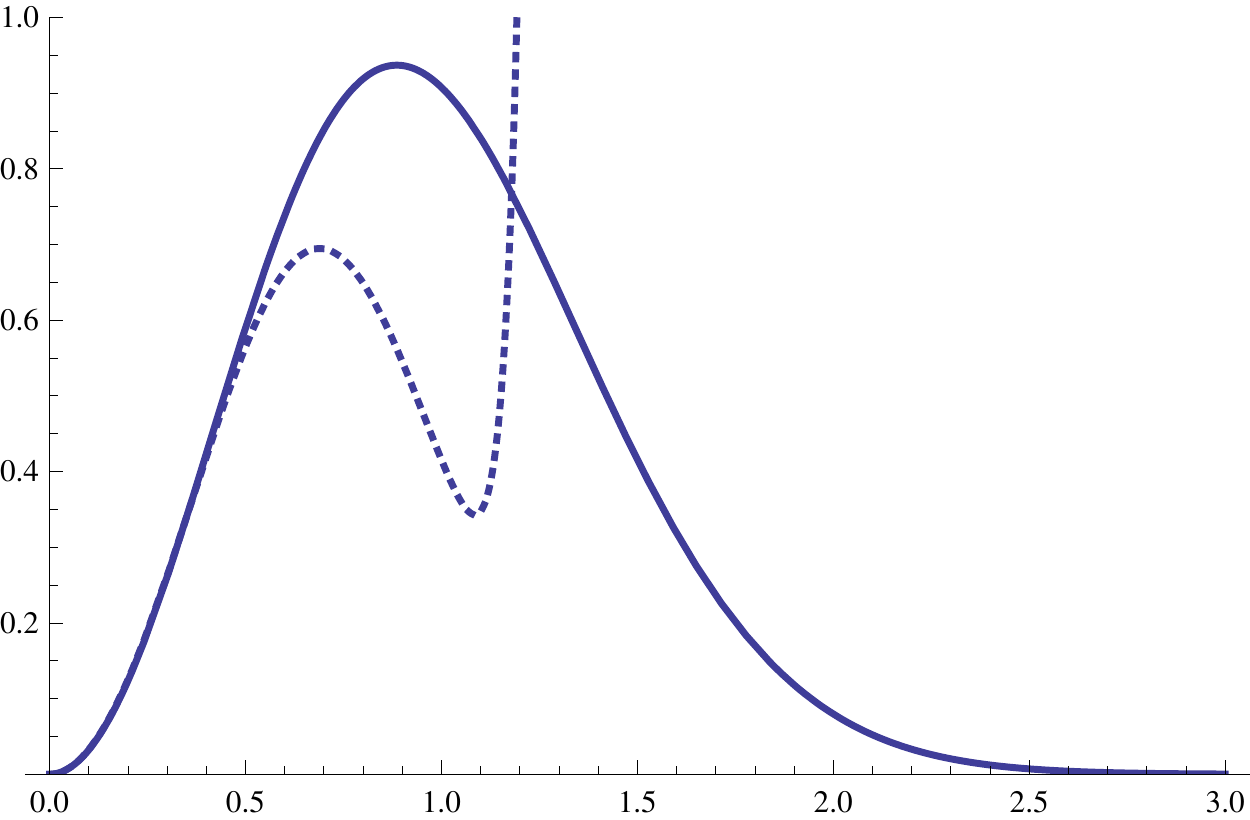}
\caption{The GUE Wigner surmise $p(x)$ (solid) v. the power series approximation to the Gaudin distribution (dotted)}\label{F:plot}
\end{center}
\end{figure}

\begin{table}
\begin{center}
\renewcommand{\arraystretch}{1.25}
\begin{tabular}{ r r r } 
$T$ & $N(T)$ & proportion\\ \hline
$500000$ & $818414$ & $0.02375$ \\
$1000000$ & $1747145$ & $0.02584$ \\
$1500000$ & $2717515$ & $0.02695$ \\
$2000000$ & $3714925$ & $0.02773$ \\
$2500000$ & $4732442$ & $0.02828$ \\
$3000000$ & $5765983$ & $0.02873$ \\
$3500000$ & $6812849$ & $0.02911$ \\
$4000000$ & $7871121$ & $0.02948$ \\
$4500000$ & $8939367$ & $0.02976$ \\
$4992381.$ & $10000000$ & $0.02998$\\
&&
\end{tabular}
\caption{Proportion of zeros at height $T$ satisfying (\ref{close})}\label{Ta:zerodata}
\end{center}
\end{table}

\section*{Numerical Experiment}  In \emph{Mathematica}, we have access to $10^7$  zeros of $\zeta(s)$ up to height $4.992\cdot 10^6$.  We can easily use this data to determine the proportion of zeros up to heights $T<4.992\cdot 10^6$ which satisfy (\ref{close}), see Table \ref{Ta:zerodata}.  The fact that the proportion of small gaps is so much less than $0.11$ is not surprising; the function $\zeta(s)$ approaches its asymptotic behavior very slowly.   We do observe in Table \ref{Ta:zerodata} that the observed proportion of close zeros is monotonically increasing.   However, this means that to use the bound (\ref{Eq:contra}) for $D$ with corresponding $T$ in the above range, we would need to replace $0.11$ in (\ref{Eq:C}) with the corresponding entry in the right column of Table \ref{Ta:zerodata}.

\section*{Application}
 For $-D<0$ any fundamental discriminant with $g$ prime factors, genus theory tells us the following.  The principal genus $\mathcal P(-D)$ fits into an exact sequence with the class group $\mathcal C(-D)$:
\[
 \mathcal P(-D)\overset{\text{def.}}=\mathcal C(-D)^2 \hookrightarrow \mathcal C(-D) \twoheadrightarrow \mathcal C(-D)/ \mathcal C(-D)^2 \simeq (\mathbb Z/2)^{g-1}.
\]
The class group $\mathcal C(-D)$ has order $h(-D)=p(-D)\cdot 2^{g-1}$.  The genera of forms are exactly what one can detect about the range via congruence conditions modulo $D$.  For example,with $-D=-39$  there are four classes of forms falling into two genera.  In a shorthand notation we write them as $\{1,1,10\}$, $\{3,3,4\}$, and $\{2,\pm1,5\}$.  A prime $p\ne 3,13$ can be written as
\begin{align*}
&\left.
\begin{array}{c}
p=x^2+xy+10y^2\\
\text{or}\\
p=3x^2+3xy+4y^2
\end{array}
\right\}&
&\Leftrightarrow &
&p\equiv 1,4,10,16,22,25 \bmod 39,\\
&\quad p=2x^2\pm xy+5y^2&
&\Leftrightarrow &
&p\equiv 2,5,8,11,20,32 \bmod 39.
\end{align*}
Classifying discriminants with a given number of classes per genus is the problem Gauss mentioned in the quote above.
Indeed, the discriminants for which there is one class per genus (and hence for which congruence conditions alone determine which form represents a given prime) are know only under the assumption of GRH; unconditionally there is at most one other than the known 65 examples, see \cite{W}.

Oesterl\'{e}, \cite{Oes} observed that the Goldfeld-Gross-Zagier lower bound (\ref{Eq:GGZ}) was too weak to say anything about discriminants with one class per genus, i.e. $p(-D)=1$.  In this section we apply the conditional bound
(\ref{Eq:contra}) to the general problem of bounding $p(-D)$ from below.

From Lemma 2 of \cite{HS} we have that
\begin{equation}\label{oldineq}
2^g<\exp(\log D \log(2)/W(\log D ))=D^{\log(2)/W(\log(D))},
\end{equation}
where $W(x)$ denotes the Lambert $W$-function, i.e. the inverse function of $f(w)=w\exp(w)$ (see  \cite{E}, \cite[p. 146 and p. 348, ex 209]{PS}).  For $x\ge0$ it is positive, increasing, and concave down.  The Lambert $W$-function is also sometimes called the product log, and is implemented as \texttt{ProductLog} in \emph{Mathematica}.  As $x\to\infty$,
\[
W(x)\sim \log(x)-\log\log(x).
\]

Thus the bound  (\ref{Eq:contra}) gives that
\begin{equation}\label{Eq:pboun}
2\cdot \frac{D^{1/2-\log(2)/W(\log(D))}}{\log(D)^{74}}\le p(-D).
\end{equation}

\end{document}